\documentclass[a4paper,11pt]{article}

\usepackage{amsmath,amssymb,amsthm,url,amsfonts}   
\usepackage{comment}
\usepackage{mathrsfs}  
\usepackage{mathtools}
\usepackage{bbm}
\usepackage{mathabx}    
\usepackage{color}
\usepackage{enumerate}  
\usepackage[shortlabels]{enumitem}
\usepackage{graphicx}
\usepackage{hyperref}

{\theoremstyle{plain}
	\newtheorem{definition}{Definition}[section]
	\newtheorem{lemma}[definition]{Lemma}

	\newtheorem{proposition}[definition]{Proposition}
	
} {\theoremstyle{definition}

}

\renewcommand{\mathbb}{\mathbbm}                     
\renewcommand{\epsilon}{\varepsilon}                 
\renewcommand{\phi}{\varphi}

\renewcommand{\le}{\leqslant}
\renewcommand{\ge}{\geqslant}

\newcommand{\origsetminus}{} \let\origsetminus=\setminus  
\renewcommand{\setminus}{\!\origsetminus\!}

\newcommand{\origfoo}{} \let\origfoo=\sqrt           
\renewcommand{\sqrt}[1]{\origfoo{#1}\;}

\newcommand{\Z}{{\mathcal Z}}

\DeclareMathOperator{\R}{{\mathbb R}}                
\DeclareMathOperator{\C}{{\mathbb C}}                

\newcommand{\ud}{{\rm d}}

\DeclareMathOperator{\sinc}{sinc}

\allowdisplaybreaks

\title{The paper ``On the constant in a transference inequality for the vector-valued Fourier transform'' revisited}
\author{Dion Gijswijt \& Jan van Neerven}
\date\today

\begin{document}
\maketitle

\begin{abstract}
The standard proof of the equivalence of Fourier type on $\mathbb R^d$ and on the torus $\mathbb T^d$
is usually stated in terms of an implicit constant which can be expressed in terms of the global minimiser of the functions
\[
f_r(x)=\sum_{m\in\mathbb{Z}}\left|\frac{\sin(\pi(x+m))}{\pi(x+m)}\right|^{2r},\qquad x\in [0,1], \ r\ge 1.
\]
The aim of this note is to provide a short proof of a result of the authors
\cite{GvN} which states that each $f_r$ takes a global minimum at the point $x = \frac12$.
\end{abstract}

\section*{Introduction}

For $x\in\mathbb{R}$ define
\(
\sinc(x):=(\sin x)/x,
\)
with the understanding that $\sinc(0)=1$, and
\(
h(x):=\sinc^2(\pi x).
\)
For $r\ge 1$ and $x\in \mathbb R$ set
\[
f_r(x):=\sum_{m\in\mathbb{Z}} (h(x+m))^r
=\sum_{m\in\mathbb{Z}} \left|\frac{\sin(\pi(x+m))}{\pi(x+m)}\right|^{2r}.
\]
The series converges absolutely for every $r\ge 1$ since $0\le h(x)\lesssim  1/x^2$ as $|x|\to\infty$. See Figure \ref{fig1}, which is taken from \cite{GvN}.
\begin{figure}
\begin{center}
\includegraphics[width=7cm]{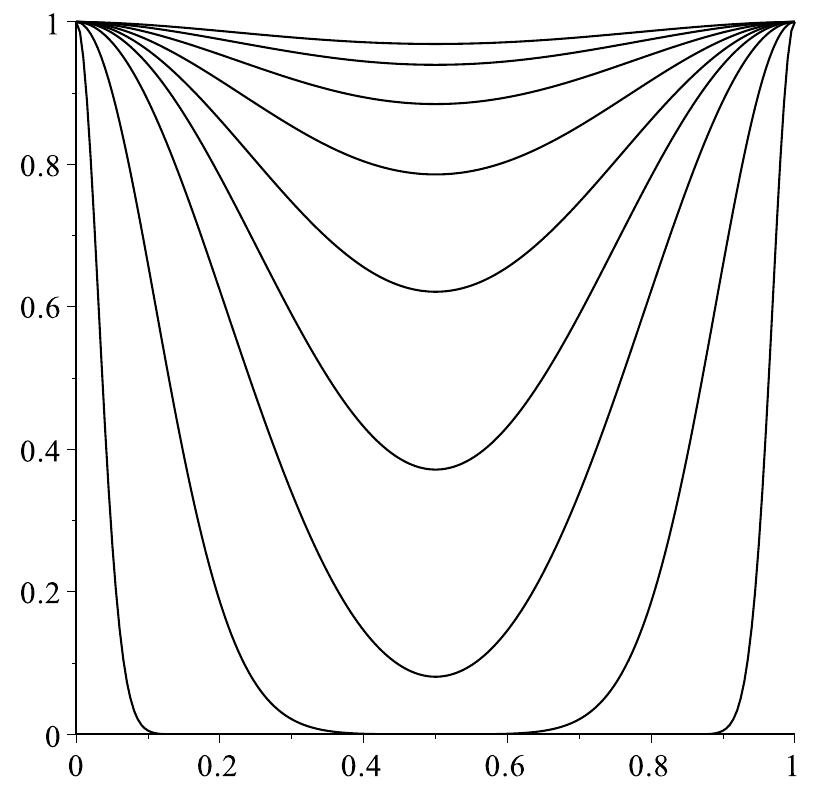}
\end{center}
\caption{A plot of $f_r$, where $r=1.02^{k}$ for $k=1,2,4,\dots,256$.\label{fig1}}
\end{figure}
This function arose the the discussion in \cite{HNVW1} of the vector-valued Fourier transform as follows.  Let $X$ be a complex Banach space. The {\em Fourier transform} of a function $f\in L^1(\R^d;X)$ is defined
by $$\mathscr{F}_{\R^d}f(\xi) := \int_{\R^d} e^{-2\pi ix\cdot \xi}f(x)\,{\rm d}x, \quad \xi\in \R^d.$$
Likewise, the {\em Fourier transform} of a function $f\in { L^1(\mathbb T^d;X)}$ is defined by
$$ \mathscr{F}_{ \mathbb T^d}f({ k}) := { \int_{\mathbb T^d}} e^{-2\pi i k\cdot t} f{(t)\,{\rm d}t}, \quad { k\in \Z^d}.$$
For any $d\ge 1$ and $p\in (1, 2]$, writing $\frac1p+\frac1q = 1$,
the following assertions are equivalent:
\begin{enumerate}
\item[(i)]
$\mathscr{F}_{\R^d}$ extends to a bounded operator from $L^p(\R^d;X)$ into $L^q(\R^d;X)$;
\item[(ii)]
$\mathscr{F}_{\mathbb T^d}$ extends to a bounded operator from $L^p(\mathbb T^d;X)$ into $\ell^q(\Z^d;X)$.
\end{enumerate}
Denoting the norms of these extensions by $\varphi_{p,X}(\R^d)$ and $\varphi_{p,X}(\mathbb T^d)$, the following quantitive bounds were proved in \cite[Section 2.4.b]{HNVW1}:
\begin{align*}
   \varphi_{p,X}(\R^d)  \leq \varphi_{d,p,X}(\mathbb T^d) \leq C_{q}^{-d/q} \varphi_{p,X}(\R^d), \end{align*}
where $C_{q}$ is the global minimum of the periodic function
$$ x\mapsto \sum_{m \in \Z} \Big|\frac{\sin(\pi(x+m))}{\pi(x+m)}\Big|^{q}, \quad x\in\R.$$
The following proposition was subsequently shown in \cite{GvN}.

\begin{proposition}\label{prop}
For every real number $r\ge 1$, the function $f_r:[0,1]\to\mathbb{R}$ has a global minimum at $x=\tfrac12$.
\end{proposition}
The proof in \cite{GvN} is elementary, but long and intricate. The purpose of the present note is to present a substantially simpler proof of this result.

The relevance of Proposition \ref{prop} to the theory of vector-valued Fourier analysis is largely historical. An alternative proof of the above equivalence of Fourier type transference constants, with optimal constants, is given in \cite[Proposition 13.1.1]{HNVW3}, and its proof does not involve a minimisation problem. More precisely, it is shown there that
\begin{align*}
   \varphi_{p,X}(\R^d)  \leq \varphi_{d,p,X}(\mathbb T^d) \leq \frac1{\varphi_{d,p;\C}(\R^d)} \varphi_{p,X}(\R^d).
   \end{align*}
It is known that $$\varphi_{d,p;\C(\R^d)} = (p^{1/p}/ q^{1/q})^d;$$ the bound $\varphi_{d,p;\C(\R^d)} \ge (p^{1/p}/ q^{1/q})^d$ is obtained by computing the $L^p$-norms of $x\mapsto e^{-\pi|x|^2}$. We refer to the Notes of \cite[Chapter 13]{HNVW3} for further discussion and bibliographic references.

\section{Proof of Proposition \ref{prop}}

For $r=1$ one has the following elementary identity.

\begin{lemma}\label{lem:r=1} $\sum_{m\in\mathbb{Z}} h(x+m) = 1$.
\end{lemma}
\begin{proof}
For fixed $x\in\mathbb{R}$ define
\(g_x(t):=e^{2\pi i x t}\) for $t\in[0,1]$.
If $x$ is an integer and $m=x$, the $m$-th Fourier coefficient of $g_x$ is given by
$$
\widehat g_x(m)
=\int_0^1 e^{2\pi i(x-m)t}\,\ud t = \int_0^1 1\,\ud t= 1,
$$
so $|\widehat g_x(m)|^2= 1 = h(0) = h(x-m)$ in that case.
In all other cases,
\[
\widehat g_x(m)
=\int_0^1 e^{2\pi i(x-m)t}\,\ud t
=\frac{e^{2\pi i(x-m)}-1}{2\pi i(x-m)}
=e^{\pi i(x-m)}\,\frac{\sin(\pi(x-m))}{\pi(x-m)}.
\]
Hence $|\widehat g_x(m)|^2= \sinc^2(\pi(x-m)) = h(x-m)$.

Parseval's identity now gives
\[
\sum_{m\in\mathbb{Z}} h(x-m)=
\sum_{m\in\mathbb{Z}} |\widehat g_x(m)|^2=\|g_x\|_{L^2(\mathbb{T})}^2=\int_0^1 |e^{2\pi i x t}|^2\,\ud t=\int_0^1 1\,\ud t=1.
\]
Replacing $m$ by $-m$ gives the result.
\end{proof}

Let $a,b\ge 0$ and $r\ge 1$. By the convexity of $t\mapsto t^r$ on $\mathbb R_+$ we have
\(
\frac{a^r+b^r}{2}\ge \bigl(\frac{a+b}{2}\bigr)^r
\)
and therefore
\[
a^r+b^r \ge 2^{1-r}(a+b)^r,
\]
with equality if and only if $a=b$.
Applying this to $a_m = h(x+m)$ and $b_m = h(x-(m+1))$, with $m\ge 0$, we obtain
\begin{equation}\label{eq:convexity-step}
f_r(x)=\sum_{m\ge 0}\Big(h(x+m)^r+h(x-(m+1))^r\Big)
\ge 2^{1-r}\sum_{m\ge 0} (s_m(x))^r,
\end{equation}
where for $m=0,1,2,\dots$ we define
\[
s_m(x):=h(x+m)+h\bigl(x-(m+1)\bigr).
\]
At $x=\tfrac12$ we have
$$h(\tfrac12+m)=  \left(\frac{\sin(\pi (\tfrac12+m))}{\pi (\tfrac12+m)}\right)^2 =
\left(\frac{\sin(\pi (\tfrac12-(m+1)))}{\pi (\tfrac12-(m+1))}\right)^2 =
h(\tfrac12-(m+1)),$$ which gives the equality
\begin{equation}\label{eq:fr-half}
f_r(\tfrac12)=2^{1-r}\sum_{m\ge 0} s_m(\tfrac12)^r.
\end{equation}
Therefore, to prove that $f_r(x)\ge f_r(\tfrac12)$ it suffices to prove
\begin{equation}\label{eq:s-sum-goal}
\sum_{m\ge 0} (s_m(x))^r \ge \sum_{m\ge 0} s_m(\tfrac12)^r.
\end{equation}

The key lemma is closely related to Lemma 3 of \cite{GvN}.

\begin{lemma}[One-crossing implies convex dominance]\label{lem:one-crossing}
Let $g:\mathbb{R}_+\to\mathbb{R}_+$ be nondecreasing and convex. Let $t\ge 0$, and let $x_1,\dots,x_n\ge 0$ and $y_1,\dots,y_n\ge 0$ satisfy
\begin{enumerate}[\rm(i)]
\item $\sum_{k=1}^n x_k=\sum_{k=1}^n y_k$;
\item for every $k\in\{1,\dots,n\}$,
$x_k\le y_k$ if
$y_k<t$, and $x_k\ge y_k$ if $y_k\ge t$.
\end{enumerate}
Then $$\sum_{k=1}^n g(x_k)\ge \sum_{k=1}^n g(y_k).$$
\end{lemma}

\begin{proof}
If $x_k = y_k$ for all $1\le k\le n$ there is nothing to prove, so we may assume that $x\neq  y$, where
$x = (x_1,\dots,x_n)$ and $y = (y_1,\dots,y_n)$.

\smallskip
\emph{Step 1} --
Let $a,b > 0$ with $a\le b$, and let $\delta\in(0,a]$. For a convex function $g:\mathbb{R}_+\to\mathbb{R}$ one has
\[
\frac{g(a)-g(a-\delta)}{\delta}\le \frac{g(b+\delta)-g(b)}{\delta}
\]
and therefore
\begin{equation}\label{eq:transfer}
g(a-\delta)+g(b+\delta)\ge g(a)+g(b).
\end{equation}

{\em Step 2} --
Let $L:=\{k:\ y_k<t\}$ and $H:=\{k:\ y_k\ge t\}$. By assumption we have
$x_k\le y_k$ for $k\in L$, and $x_k\ge y_k$ for $k\in H$,
and $\sum_{k=1}^n x_k=\sum_{k=1}^n y_k$.
By assumption (i) we have $y_k-x_k\ge 0$ for $k\in L$, $x_k-y_k\ge 0$ for $k\in H$, and
\[
\sum_{k\in L} y_k-x_k  = \sum_{k\in H} x_k-y_k.
\]

In the next two steps of the proof,
we will obtain the vector $x=(x_1,\dots,x_n)$ from $y=(y_1,\dots,y_n)$ by finitely many ``mass transfers''
from indices in $L$ to indices in $H$, in such a way that each transfer does not decrease $\sum g(\cdot)$.

\smallskip
\emph{Step 3} -- Set $z^{(0)}:=y$. Proceeding by induction, suppose $z^{(0)},\dots, z^{(m)}\in \R_+^n$ have been defined in such a way that:
\begin{itemize}
\item[\rm(a)] $\sum_{j=1}^n z_j^{(k)}=\sum_{j=1}^n y_j=\sum_{j=1}^n x_j$ for all $k=0,\dots,m$;
\item[\rm(b)] $x_i\le z_i^{(k)} \le y_i$ for all $i\in L$ and $y_j\le z_j^{(k)} \le x_j$ for all $j\in H$ and $k=0,\dots,m$.
\end{itemize}
If $z^{(m)}= x$ we terminate the procedure.
We claim that if $z^{(m)}\neq x$, then there exists $i\in L$ and $j\in H$ such that
$z_i^{(m)}>x_i$ and $z_j^{(m)}<x_j$.
Indeed, since z$^{(m)}\neq x$, and $\sum_i z_i^{(m)} = \sum_i x_i$, there are indices $i, j$ such that $z_i^{(m)} > x_i$ and $z_j^{(m)} < x_j$.
By condition (b) it follows that $i\in L$ and $j\in H$. This proves the claim.

Put
\[
\delta:=\min\{\,z^{(m)}_i-x_i,\ x_j-z^{(m)}_j\,\}>0,
\]
and define $z^{(m+1)}\in\mathbb{R}_+^n$ by
\[
z^{(m+1)}_i:=z^{(m)}_i-\delta,\qquad z^{(m+1)}_j:=z^{(m)}_j+\delta,\qquad
z^{(m+1)}_k:=z^{(m)}_k\ (k\notin\{i,j\}).
\]
Then (a) and (b) hold for all $k = 0,\dots, m+1$.
Since either $z^{(m+1)}_i=x_i$ or $z^{(m+1)}_j=x_j$, after finitely many steps (say, after $M$ steps)
the procedure terminates with $z^{(M)}=x$.

\smallskip
\emph{Step 4} --
At every step in the above procedure we have $z^{(m)}_i\le y_i<t\le y_j\le z^{(m)}_j$ for $i\in L$ and $j\in H$.
Applying \eqref{eq:transfer} with $a=z^{(m)}_i$, $b=z^{(m)}_j$, we obtain
\[
g(z^{(m+1)}_i)+g(z^{(m+1)}_j)
\;=\;
g(z^{(m)}_i-\delta)+g(z^{(m)}_j+\delta)
\;\ge\;
g(z^{(m)}_i)+g(z^{(m)}_j).
\]
All other coordinates are unchanged, so
\[
\sum_{k=1}^n g(z^{(m+1)}_k)\ge \sum_{k=1}^n g(z^{(m)}_k).
\]
Iterating from $m=0$ to $M-1$, we obtain
\[
\sum_{k=1}^n g(x_k)=\sum_{k=1}^n g(z^{(M)}_k)\ge \sum_{k=1}^n g(z^{(0)}_k)=\sum_{k=1}^n g(y_k),
\]
which is the desired conclusion.
\end{proof}

The elementary proof of the following lemma is given in \cite[Lemma~5(i)]{GvN}:

\begin{lemma}\label{lem:m0}
The function $x\mapsto s_0(x)=h(x)+h(x-1)$ has a global minimum on $[0,1]$ at $x=\tfrac12$.
\end{lemma}

We continue with a simplified proof of \cite[Lemma 5(ii)]{GvN}:

\begin{lemma}\label{lem:mge1}
For every $m\ge 1$, the function $x\mapsto s_m(x)$ has a global maximum on $[0,1]$
at $x=\tfrac12$.
\end{lemma}

\begin{proof}
Using the identity $\sin(\pi(x+m))=(-1)^m\sin(\pi x)$ we may write, for $m\ge 0$,
\begin{equation}\label{eq:s-explicit}
s_m(x)=\frac{\sin^2(\pi x)}{\pi^2}\Bigl(\frac{1}{(m+x)^2}+\frac{1}{(m+1-x)^2}\Bigr),\qquad x\in[0,1].
\end{equation}
Set $u:=x-\tfrac12\in[-\tfrac12,\tfrac12]$ and $D:=m+\tfrac12\ge \tfrac12$. Then $\sin(\pi x)=\cos(\pi u)$ and
$m+x=D+u$, $m+1-x=D-u$. Neglecting the factor $\pi^{-2}$ for the moment, set
\[
\Phi_D(u):=\cos^2(\pi u)\Bigl(\frac{1}{(D+u)^2}+\frac{1}{(D-u)^2}\Bigr),
\qquad u\in[-\tfrac12,\tfrac12].
\]
This function is even in $u$, and $u=0$ is a critical point. For $u\in(0,\tfrac12)$ we compute the logarithmic derivative
\begin{equation}\label{eq:log-derivative}
\frac{\rm d}{{\rm d}u}\log \Phi_D(u)
= -2\pi\tan(\pi u) +\frac{2u}{D^2+u^2} +\frac{4u}{D^2-u^2}.
\end{equation}

Fix $m\ge 1$ and $u\in(0,\tfrac12)$. Then $D\ge \tfrac32$. Since $\tan(\pi u)\ge \pi u$ for $u\ge 0$, the first term in
\eqref{eq:log-derivative} satisfies $-2\pi\tan(\pi u)\le -2\pi^2 u$.
Moreover,
\[
\frac{2u}{D^2+u^2}+\frac{4u}{D^2-u^2}
\le
\frac{2u}{D^2-u^2}+\frac{4u}{D^2-u^2}
=
\frac{6u}{D^2-u^2}
\le
\frac{6u}{D^2-\tfrac14}.
\]
As $D\ge \tfrac32$, we have $D^2-\tfrac14\ge 2$, so the rational part of \eqref{eq:log-derivative} is bounded above by $3u$. Therefore,
\[
\frac{\rm d}{{\rm d}u}\log \Phi_D(u)\le -2\pi^2 u+3u = u(3-2\pi^2)<0.
\]
It follows that $\Phi_D(u)$ is strictly decreasing for $u\in (0,\tfrac12)$, and since $\Phi_D$ is even we infer that $u=0$ is the unique global maximum on
$[-\tfrac12,\tfrac12]$. Translating back, $x=\tfrac12$ is the global maximum of $s_m$ on $[0,1]$.
\end{proof}

\begin{proof}[Proof of Proposition \ref{prop}]
For integers $m\ge 0$ set $y_m:=s_m(\tfrac12)$ and $x_m:=s_m(x)$.
From
\[
y_m=s_m\!\left(\tfrac12\right)
=2h\!\left(m+\tfrac12\right)
=2\left(\frac{\sin(\pi(m+\tfrac12))}{\pi(m+\tfrac12)}\right)^2
=\frac{8}{\pi^2(2m+1)^2},
\qquad m\ge 0,
\]
we infer that the sequence $(y_m)_{m\ge0}$ is strictly decreasing.

By Lemma~\ref{lem:m0} we have $x_0\ge y_0$, and by Lemma~\ref{lem:mge1} we have $x_m\le y_m$ for all $m\ge 1$.
Moreover, by Lemma \ref{lem:r=1},
\begin{align}\label{eq:Parseval}
\sum_{m\ge 0} x_m = \sum_{m\ge 0} y_m = 1.
\end{align}
Fix $t\in \R$ with $y_1<t<y_0$; then $y_0\ge t$ and
$y_m\le y_1<t$ for all $m\ge1$.
For $N\ge 0$ set
\[
X_N:=\sum_{m\ge N+1} x_m,\qquad Y_N:=\sum_{m\ge N+1} y_m,
\]
and define
$\mathbf{x}^{(N)}:=(x_0,x_1,\dots,x_N,X_N)$ and $\mathbf{y}^{(N)}:=(y_0,y_1,\dots,y_N,Y_N).$
By \eqref{eq:Parseval},
\[
\sum_{k=0}^{N} x_k + X_N = \sum_{k=0}^{N} y_k + Y_N = 1.
\]
Moreover, for $m\ge1$ we have $x_m\le y_m$, hence also $X_N\le Y_N$ for every $N\ge 0$.
Since $Y_N\to 0$ as $N\to\infty$, there exists $N_0\ge 0$ such that $Y_N<t$ for all $N\ge N_0$.
Fix such an $N\ge N_0$. Then $y_0\ge t$, $y_m<t$ for $1\le m\le N$, and also $Y_N<t$.
The pointwise inequalities $x_0\ge y_0$, $x_m\le y_m$ for $1\le m\le N$, and $X_N\le Y_N$ show that
$\mathbf{x}^{(N)}$ and $\mathbf{y}^{(N)}$ satisfy the assumptions of Lemma~\ref{lem:one-crossing}
(with the same threshold $t$). Applying this lemma to $g(u)=u^r$, we obtain
\begin{align}\label{eq:approxN}
\sum_{m=0}^{N} x_m^r + X_N^r \ge \sum_{m=0}^{N} y_m^r + Y_N^r.
\end{align}
Letting $N\to\infty$ (with $N\ge N_0$) we have $X_N\to 0$ and $Y_N\to 0$, hence also $X_N^r\to 0$ and $Y_N^r\to 0$.
Moreover, the partial sums satisfy
$$\sum_{m=0}^{N} x_m^r \rightarrow \sum_{m\ge 0} x_m^r, \qquad
\sum_{m=0}^{N} y_m^r \rightarrow \sum_{m\ge 0} y_m^r$$ as $N\to\infty$,
since both sums $\sum_{m\ge 0} x_m^r$ and $\sum_{m\ge 0} y_m^r$ converge. Therefore, passing to the limit in
\eqref{eq:approxN}, we obtain $$\sum_{m\ge 0} x_m^r \ge \sum_{m\ge 0} y_m^r,$$
which is \eqref{eq:s-sum-goal}.
\end{proof}

\noindent {\em Acknowledgment} -- The present proof was found with the help of GPT5.2.

\end{document}